\documentclass[12pt]{article}
\usepackage{latexsym,color,amsmath,amsthm,amssymb,amscd,amsfonts}
\usepackage{scalefnt}
\usepackage[font=small,labelfont=bf]{caption}
\usepackage{float}
\usepackage{graphicx}
\usepackage{amsmath}
\usepackage{amsmath}
\usepackage{relsize}
\usepackage{mathrsfs}
\usepackage{upgreek}
\usepackage{sectsty}
\usepackage{bbm}

\newtheoremstyle{nonum}{}{}{\itshape}{}{\bfseries}{.}{ }{\thmnote{#3}}

\newtheorem{thm}{Theorem}[section]
\newtheorem*{thm*}{Theorem}

\newtheorem{lem}[thm]{Lemma}
\newtheorem*{lem*}{Lemma}
\newtheorem{rem}[thm]{Remark}
\newtheorem*{rem*}{Remark}
\newtheorem{prop}[thm]{Proposition}
\newtheorem*{prop*}{Proposition}

\newtheorem*{conj*}{Conjecture}

\newtheorem*{definition*}{Definition}

\newtheorem*{fact*}{Fact}

\newtheorem*{rems*}{Remarks}

\theoremstyle{nonum}

\newcommand{\R}{\mathbb R}

\newcommand{\N}{\mathbb N}

\def\vol{{\rm Vol}}
\def\grad{{\nabla}}

\def\eps{{\varepsilon}}

\def\conv{{\rm conv}}

\newcommand{\capsilon}[2]{C(#1,#2)}

\begin{document}

\title{On the Many Faces of Easily Covered Polytopes}
\author{Dan I. Florentin, Tomer Milo}
\date{}
\maketitle

\begin{abstract}
Assume that $rB_{2}^{n} \subset P$ for some polytope $P \subset \R^n$, where $r \in (\frac{1}{2},1]$. Denote by $\mathcal{F}$ the set of facets of $P$, and by $N=N(P,B_2^n)$ the covering number of $P$ by the Euclidean unit ball $B_2^n$. We prove that if $\log N \le\frac{n}{8}$, then
\[
|\mathcal{F}| \ge \left(
\frac{1}{ 2\left(1 - r \sqrt{1-\frac{4\log N}{n}}\right) }
\right)^{\frac{n-1}{2}}.
\]
\end{abstract}

\section{Introduction and the main result}
A polytope $P \subset \R^n$ is the convex hull of finitely many points. The relationship between the geometry and combinatorics of polytopes in high dimensions has been studied extensively in the literature. In \cite{MS}, Milman and Szarek studied the connection between the number of faces of a polytope $P\subset B_2^n$, its inradius, and its covering number by $B_{2}^{n}$, namely
\[N(P,B_2^n)= \min\left\{N: \exists x_{1},...,x_{N} \in \R^n, P \subset \bigcup_{i=1}^{N} \left( x_{i} + B_{2}^{n} \right)\right\}
. \]
They formulated the following ``geometric lemma''.
\begin{conj*}[Milman-Szarek]\label{conj-MS}
Let $r\in(0,1]$. Then there exists $c=c(r)$ such that for every $n\in\N$ and for every polytope $P=\conv(V)\subset\R^n$ containing $rB_2^n$, then
\[
N(P,B_2^n) < e^{cn} \implies |V| > e^{cn}.
\]
In particular, the covering number $N$ and number of vertices $|V|$ cannot both be sub exponential (in the dimension).
\end{conj*}
Milman and Szarek showed that this conjecture implies the ``duality conjecture'' regarding covering numbers, and while the latter has since been proved (see \cite{AMS}), the former is still open.
In this note we prove an analogue of the conjectured geometric lemma, where the set of vertices $V$ is replaced by the set of facets $\mathcal{F}$, i.e. the set of $(n-1)$-dimensional faces of $P$. We are not aware of some duality argument which allows us to replace facets by vertices, towards a proof of Milman and Szarek's geometric lemma.

\begin{thm}\label{thm: main theorem}
Let $n\ge 3$ and let $P\subset \R^n$ be a polytope containing $rB_2^n$, for some $\frac{3\sqrt{3}}{\sqrt{n}}\le r \le 1$. Then
\[
3\le N = N(P,B_2^n) < e^{\frac{n}8} \implies
|\mathcal{F}| >
\left(
\frac{1}{2\left(1-r\sqrt{1-\frac{4\log N}{n}}\right)}
\right)^{\frac{n-1}{2}},
\]
where $\mathcal{F}$ is the set of facets of $P$.
\end{thm}
Note that for the maximal value $r=1$, the inequality $1-\sqrt{1-\eps} \le \eps$ yields
\[ |\mathcal{F}| \ge
\left( \frac{n}{8\log N}\right)^{\frac{n-1}{2}}.
\]
In fact, when $r$ approaches $1$ and $\log N = o(n)$, we get a super-exponential lower bound (see Remark \ref{rem:super-exp} below).
We would like to highlight the similarity between our result and the following well known fact, which states that if $rB_{2}^{n} \subset P \subset B_{2}^{n}$ then the set of facets is large. Throughout this text, $\sigma$ denotes the unique rotation invariant probability measure on $S^{n-1}$. The cap around $u$ at height $h$ (for $\theta \in S^{n-1}, h\in (0,1)$) is defined to be
\[
\capsilon{\theta}{h}= \{u \in S^{n-1}: \langle u, \theta \rangle \ge h\}.
\]
\begin{prop}\label{prop: many faces of polytope close to ball}
Let $n\ge 3$, let $P \subset \R^n$ be a polytope, and denote by  $\mathcal{F}$ and $V$ the sets of facets and vertices of $P$, respectively. 
Assume that $rB_{2}^{n} \subset P \subset B_{2}^{n}$ for some $\frac1{\sqrt{n-1}} \le r \le 1$. Then
\[ \min\{ |\mathcal{F}|, |V| \} \ge
\left(\frac{1}{2(1-r)}\right)^{\frac{n-1}{2}}. \] 
\end{prop}
For the sake of completeness we include a proof of this classical result, based on estimating the measure of a spherical cap of height $r$, for $r \in (\frac{1}{2},1)$ (and a union bound). We prove the estimate itself in Lemma \ref{asymptotic measure of a small cap} below.
\begin{proof} Since $rB_{2}^{n} \subset \conv(V)$, $V$ must intersect every $r$-cap of the unit sphere $S^{n-1}$, that is,
for every $\theta \in S^{n-1}$ there exists $v\in V$ with $\langle \theta, v \rangle \ge r$, i.e. $\theta \in \capsilon{v}{r}$. Thus
\[
S^{n-1} = \bigcup_{v \in V} \capsilon{v}{r}.
\]
Applying a union bound, we obtain
\[
1 = \sigma(S^{n-1}) =
\sigma \left( \bigcup_{v \in V}  \capsilon{v}{r} \right) \le
|V|\cdot \sigma\left( \capsilon{v}{r} \right).
\]
Using Lemma \ref{asymptotic measure of a small cap} to estimate the $\sigma$-measure of a cap, we get
\[
|V| \ge \frac{1}{\sigma\left( \capsilon{v}{r} \right)} \ge
\frac{r\sqrt{n-1}}{\left( 2(1-r) \right)^{\frac{n-1}{2}}} \ge
\left( \frac{1}{2(1-r)} \right)^{\frac{n-1}{2}}.
\]

\noindent This proves Proposition \ref{prop: many faces of polytope close to ball} for the vertices; the same bound applies for the number of facets $\mathcal{F}$ by duality, i.e. by considering $Q=r P^{\circ}$. 
\end{proof}

\noindent \textbf{Notation.} Throughout the text, we will use the following notations. For $n \in \mathbb{N}$, by a convex body $K \subset \mathbb{R}^{n}$ we mean a compact convex set with non-empty interior. The convex hull of a set $V\subset \R^n$ is denoted by $\conv(V)$, and we shall usually consider a polytope $P=\conv(V)$, i.e. the convex hull of a finite set $V$ (the vertices of $P$). The set of facets of $P$ (i.e. the set of $(n-1)$-dimensional faces of $P$) will be denoted by $\mathcal{F}$. The Euclidean ball in $\R^n$ (centered at $0$) and its boundary, are denoted by $B_{2}^{n}$ and $S^{n-1}$ respectively. The Euclidean norm of a vector $x\in\R^n$ is denoted by $|x|$.
If $K\subset\R^n$ is a convex body containing $0$ in its interior, the gauge function of $K$ may be defined (on $\R^n$) as follows.
\[
\|x\|_{K} = \inf\{ t \ge 0 \,:\, x \in tK \}.
\]
If $K$ is centrally symmetric, i.e. $K =- K$, then the gauge function $\|\cdot\|_{K}$ is a norm on $\mathbb{R}^{n}$, and is strictly positive on the sphere $S^{n-1}$. Its reciprocal is called the radial function of $K$, and is defined for $\theta \in S^{n-1}$ as follows.
\[
\rho_{K}(\theta) = \sup\left\{t \ge 0 \,:\, t\theta \in K\right\} = \frac{1}{\|\theta\|_{K}},
\]
so that $K\cap \R^+\theta = [0,\rho_K(\theta)\theta]$. The support function of $K$ is defined for $\theta \in S^{n-1}$ by
\[
h_{K}(\theta) = \sup\left\{\langle x,\theta \rangle  \,:\, x \in K\right\}.
\]
Let $K,T \subset \R^n$ be convex bodies. The covering number of $K$ by $T$, denoted $N(K,T)$, is defined to be the number of translates of $T$ required to cover $K$ i.e.
\[N(K,T) = \min \left\{ N \,:\, \exists x_{1},...,x_{N} \in \mathbb{R}^{n}: K \subset \bigcup_{i=1}^{N}(x_{i} + T) \right\}. \]
In this note we only consider covering by translates of the Euclidean ball i.e. $T = B_{2}^{n}$.

\noindent {\bf Acknowledgments:} The second named author was partially supported by the ERC under the European Union’s Horizon 2020 research and innovation programme (grant agreement no. 770127), and by the Binational Science Foundation (grant no. 2020329).

\section{Proof of Theorem \ref{thm: main theorem}}
Our starting point is the following well known estimate of the measure of $\eps$-caps in $S^{n-1}$, namely sets of the form $\capsilon{u}{\eps} = \{\theta \in S^{n-1}: \langle u, \theta \rangle \ge \eps\}$.
\begin{prop}\label{large caps are small}
    Let $u \in S^{n-1}$. For every $\eps \in (0,1)$ we have that
\begin{equation*}
    \sigma(\capsilon{u}{\eps}) \leq e^{-\frac{1}{2}n\eps^{2}}.
\end{equation*}
\end{prop}
For the proof of this fact see e.g. \cite[Remark 3.1.8]{AGA}.
Next, we define the set $A_\eps$, which plays a central role in what follows.
Let $\eps \in (0,1)$. For a convex body $K \subset \mathbb{R}^{n}$ with covering number $N = N(K,B_{2}^{n})$, and $\left\{ x_i \right\}_{i=1}^N \subset \mathbb{R}^n$ such that $K \subset \bigcup_{i=1}^{N} (x_{i} + B_{2}^{n})$, we let
\begin{equation}\label{def:A-epsilon}
A_{\eps} = \{\theta \in S^{n-1}: \: |\langle \hat{x_{i}}, \theta \rangle| \leq \varepsilon, \: \forall i\in [N] \}.    
\end{equation}
Here $\hat{x} = \frac{x}{|x|} \in S^{n-1}$ for $x \neq 0$ (if $x_{i}=0$ for some $i$, we ignore it in the definition). Note that $A_{\eps}$ depends on the covering $\{x_i\}_{i=1}^{N}$; in fact $A_{\eps}$ is the set of all points in $S^{n-1}$ that are `$\eps$-orthogonal' to all the normalized centers of the covering. Proposition \ref{large caps are small} implies that $A_{\eps}$ has a large $\sigma$-measure, as long as the size of $N$ is sub exponential in the dimension $n$. This observation is the main idea behind the following proposition.

\begin{prop}\label{prop: radial function is small}
Let $K \subset \mathbb{R}^{n}$ be a convex body with $0 \in {\rm int}(K)$ and covering number $N = N(K,B_{2}^{n})$. Then for every $\eps \in (0,1)$ we have
\begin{equation*}
    \sigma
    \left(
    \left\{
    \theta \in S^{n-1} \,:\,
    \rho_{K}(\theta) \leq \frac{1}{\sqrt{1-\varepsilon^{2}}}
    \right\}
    \right) \geq
    1- 2e^{\log N-\frac{1}{2}n\varepsilon^{2}}.
\end{equation*}
\end{prop}

\begin{proof}
Fix some covering of $K$ by $N$ balls, i.e, $K \subset \bigcup_{i=1}^{N} \left( x_{i} + B_{2}^{n} \right)$, and let $A_{\eps}$ be the intersection of $N$ slabs as defined  in \eqref{def:A-epsilon}. We have
\begin{eqnarray*}
\sigma(A_{\eps}) &=& \sigma(\{\theta \in S^{n-1}: \forall i, \, |\langle \hat{x_{i}}, \theta \rangle| \leq \eps \}) =
1 -
\sigma(\{\theta \in S^{n-1}: \exists i,\, |\langle \hat{x_{i}}, \theta \rangle| > \eps \}) \\ \\
&\ge& 1 - 2N\sigma( \capsilon{\hat{x_{1}}}{\eps}).
\end{eqnarray*}
Using the estimate of a cap measure from Proposition \ref{large caps are small}, we get
\begin{equation}\label{eq:bound-for-A-eps}
\sigma(A_{\eps}) \ge 1 - 2Ne^{-\frac{1}{2}n\eps^{2}}  = 1 -  2e^{\log N -\frac{1}{2}n\eps^{2}}.
\end{equation}
Let $G(\theta,\eps) = \{ x \in \mathbb{R}^{n}: |\langle  x, \theta \rangle| > \eps |x| \}$  be the (open) double cone about the axis $\R \theta$ with angle $\alpha = \arccos(\eps)$. The set $A_\varepsilon$ consists of all points $\theta\in S^{n-1}$ for which the set $\left\{ x_i \right\}_{i=1}^N$ lies in the complement of the double cone $G(\theta, \eps)$, i.e. $\left\{ x_i \right\}_{i=1}^n \subset \R^n\setminus G(\theta,\eps)$. Thus
\begin{equation*}
\theta\in A_\eps \quad \implies \quad
K \subset \left\{ x_i \right\}_{i=1}^N + B_{2}^{n} \subset
(\R^n \setminus G(\theta,\eps)) + B_{2}^{n}.
\end{equation*}
Basic trigonometry shows that the radial function of the unbounded, closed, convex set $T=(\R^n \setminus G(\theta,\eps)) + B_{2}^{n}$ at $\theta$ is given by $\rho_T(\theta)=\frac1{\sqrt{1-\eps^{2}}}$, and by the inclusion above, $\rho_{K}(\theta) \leq \frac{1}{\sqrt{1-\varepsilon^{2}}}$ for all $\theta \in A_{\eps}$. Together with \eqref{eq:bound-for-A-eps}, this completes the proof.
\end{proof}

\begin{rem}
The proof of Proposition \ref{prop: radial function is small} does not rely on convexity of $K$. Indeed, given $\left\{ x_i \right\}_{i=1}^N$, we showed that $\rho_T$ is small on the set $A_\eps$ (which has large measure), and the only assumption on $K$ is that it is covered by $\left\{ x_i \right\}_{i=1}^N + B_{2}^{n}$.
\end{rem}
The next ingredient in the proof is the notion of the Gauss map $n_K$ associated with a convex body $K$. For a smooth convex body $K$, the function $n_{K}: \partial K \to S^{n-1}$ is defined such that $n_K(x)$ is the (unique) outer unit normal to $K$ at $x \in \partial K$.
For a general convex body $K$, the function $n_K : \partial K\to P(S^{n-1})$ is a set valued map, and $n_K(x)\subset S^{n-1}$ is called the normal cone of $K$ at $x$.
Every $x \in \partial K$ can be written uniquely as $x= \rho_{K}(\theta)\theta$ for $\theta = \hat{x}\in S^{n-1}$. Clearly. if $x\in\partial K$ and $u\in n_K(x)$, then $h_{K}(u) = \langle u, x \rangle$ (see e.g. \cite[page 115]{Sch}).
Our next step is using Proposition \ref{prop: radial function is small} to show that for every $\theta$ in the large set $A_{\eps}$,  $\theta$ must be close to every $u\in n_{K}(\rho_{K}(\theta)\theta)$. In other words, the normal cone $n_{K}(\rho_{K}(\theta)\theta)$ at $\rho_K(\theta)\theta\in\partial K$ is contained in a cap around $\theta$, whose size is determined by $\eps$ and by the inradius of $K$.
\begin{prop}\label{gauss map close to id}
Let $K \subset \mathbb{R}^{n}$ be a convex body containing $rB_{2}^{n}$ for some $r \in (0,1]$. Let $\{x_i\}_{i=1}^N$ be centers of a covering of K of size $N = N(K,B_{2}^{n})$, and let $A_\eps$ as in \eqref{def:A-epsilon}. Then for all $\theta \in A_{\eps}$, and all $u\in n_K(\rho_K(\theta)\theta)$, we have
\begin{equation*}
    \langle u, \theta \rangle \geq r\sqrt{1-\eps^{2}}.
\end{equation*}
\end{prop}

\begin{proof}
Fix $\theta \in A_{\eps}$. Note that $r\le h_K$, since $rB_2^n\subset K$. For $u\in n_K(\rho_K(\theta)\theta)$ we get
\begin{align*}
r \le h_{K}(u) =
\langle u, \rho_{K}(\theta)\theta \rangle =
\rho_{K}(\theta) \langle u, \theta \rangle \le
\frac{\langle u , \theta \rangle}{\sqrt{1-\eps^{2}}}&.\qedhere
\end{align*}
\end{proof}
The final ingredient we require is the following upper bound for the measure of a cap, which is sharp for small caps.
\begin{lem}\label{asymptotic measure of a small cap}
Let $u\in S^{n-1}$ and $h \in (0,1)$. Then
\[
\sigma(\capsilon{u}{h})
\le
\frac{(2(1-h))^{\frac{n-1}{2}}}{h\sqrt{n-1}}.
\]
\end{lem}
\begin{proof}
We compute for $u = e_{n}$. Letting $B_{2}^{n-1} := B_{2}^{n} \cap \{x_{n} = 0\}$, and defining $f:B_{2}^{n-1} \to [0,1]$ by $f(x) = \sqrt{1-|x|^2}$, we have $f^{-1}\left([h,1]\right)=\sqrt{1-h^2}B_2^{n-1}$ and
\[
\capsilon{e_n}{h} = \{(x,f(x)): f(x) \ge h \}
= \{(x,f(x)): |x|^2 + h^2 \le 1 \}.
\]
Note that $\grad{f}(x)=\frac{-x}{f(x)}$, and thus $\sqrt{1+|\grad{f}|^2 } =\frac{1}{f(x)}$,
so on the set $\sqrt{1-h^2}B_2^{n-1}$ we have $\sqrt{1+|\grad{f}|^2 }\in \left(1,\frac1h\right)$. We get 
\begin{eqnarray*}
\sigma(\capsilon{e_n}{h})
&=& 
\frac{1}{\vol_{n-1}(S^{n-1})}
\int_{\sqrt{1-h^2} B_2^{n-1}} \sqrt{1+|\grad{f}(x)|^2} dx
\\ \\ &<&
\frac{\vol_{n-1}\left(\sqrt{1-h^2} B_2^{n-1}\right)}{h\cdot\vol_{n-1}(S^{n-1})}
=
\frac{(1-h^2)^{\frac{n-1}{2}}}{h}\cdot
\frac{\vol_{n-1}\left(B_2^{n-1}\right)}{\vol_{n-1}(S^{n-1})}
\\ \\ &=&
\frac{(1-h^2)^{\frac{n-1}{2}}}{nh}\cdot
\frac{\kappa_{n-1}}{\kappa_n}=
\frac{(1-h^2)^{\frac{n-1}{2}}}{nh \sqrt{\pi}}\cdot
\frac{\Gamma\left(\frac{n}{2} + 1\right)}
{\Gamma\left(\frac{n}{2} + \frac12\right)}
\\ \\ &<&
\frac{(1-h^2)^{\frac{n-1}{2}}}{nh \sqrt{\pi}}\cdot
\sqrt{
\frac
{\Gamma\left(\frac{n}{2} + \frac32\right)}
{\Gamma\left(\frac{n}{2} + \frac12\right)}
}=
\frac{(1-h^2)^{\frac{n-1}{2}}}{nh \sqrt{\pi}}
\sqrt{\frac{n+1}{2}}
<
\frac{(1-h^2)^{\frac{n-1}{2}}}{h\sqrt{n-1}},
\end{eqnarray*}
where $\Gamma$ denotes the Gamma function, and we have used the well known identity $\Gamma(x+1)=x\Gamma(x)$ (see e.g. \cite[Chapter 5]{Handbook}). The second inequality holds since $\Gamma$ is log-convex, i.e. $\Gamma\left(x + \frac12\right)\le \sqrt{\Gamma(x)\Gamma(x + 1)}$. Since $1-h^2=(1+h)(1-h)<2(1-h)$, the proof is complete.
\end{proof}
We are now ready to prove Theorem \ref{thm: main theorem}, using Proposition \ref{gauss map close to id}, for a polytope $K = \conv(V)$, and Lemma \ref{asymptotic measure of a small cap}.

\noindent {\bf Proof of Theorem \ref{thm: main theorem}.}
Denote the set of faces of $P$ by $\mathcal{F}$, and let $\mathcal{U} = \{u_{F}: F \in \mathcal{F} \}$ be the set of (outer) unit normals to the facets. For every $\theta$, denote by $F(\theta)$ the facet of $P$ such that $\rho_{P}(\theta)\theta \in F$ (this correspondence is $\sigma$-a.e well defined on $S^{n-1}$). By Proposition \ref{gauss map close to id}, for every $\theta \in A_{\eps}$ we have
\begin{equation}\label{eq:U-is-net-for-A-eps}
    \langle u_{F(\theta)}, \theta \rangle \geq r\sqrt{1-\varepsilon^{2}},
\end{equation}
i.e, the set $\mathcal{U}$ is a net for $A_{\eps}$, meaning that caps at height $r\sqrt{1-\eps^{2}}$ around the normals $\mathcal{U}$ consist a covering of $A_{\eps}$. Using \eqref{eq:bound-for-A-eps} and a union bound we get
\begin{equation*}
1-2e^{\log N-\frac{1}{2}n\eps^{2}} \le
\sigma(A_\eps) \le
\sigma
\left(
\bigcup_{F \in \mathcal{F}}  \capsilon{u_F}{r\sqrt{1-\eps^{2}}} \right) \le
|\mathcal{F}| \cdot \sigma(\capsilon{u}{r\sqrt{1-\eps^{2}}}).
\end{equation*}
By Lemma \ref{asymptotic measure of a small cap} for $h=r\sqrt{1-\eps^2}$, we get
\[
|\mathcal{F}| \ge
\left(1-2e^{\log N-\frac{1}{2}n\eps^{2}}\right)\cdot
\left(
\frac{(2(1-r\sqrt{1-\eps^2}))^{\frac{n-1}{2}}}
{r\sqrt{1-\eps^2}\sqrt{n-1}}
\right)^{-1}.
\]
The assumption $N(K,B_2^n)<e^{\frac{n}{8}}$ implies that $\sqrt{\frac{4\log N}{n}}\in \left(0,\frac1{\sqrt{2}}\right)$. Thus we may choose $\eps=\sqrt{\frac{4\log N}{n}}$, to get $\sqrt{1-\eps^2}\ge \sqrt{\frac12}$. The assumption $3\le N$ implies that
$1-2e^{\log N-\frac{1}{2}n\eps^{2}} =
1 -\frac2N \ge \frac{1}{3}$. Combining both of these bounds with the assumption $\frac{3\sqrt{3}}{\sqrt{n}} \le r$ yields
\[
|\mathcal{F}| \ge
\frac{r\sqrt{n-1}}{3\sqrt{2}}\cdot
\left(
\frac{1}{(2(1-r\sqrt{1-\eps^2}))^{\frac{n-1}{2}}}
\right) \ge
\left(
\frac{1}{(2(1-r\sqrt{1-\eps^2}))}
\right)^{\frac{n-1}{2}},
\]
where the last inequality $\sqrt{\frac{n}{n-1}} \le \sqrt{\frac32}$ holds for all $n\ge 3$.
\qed

\begin{rem}\label{rem:super-exp}
Recall that for $r=1$, the inequality $1-\sqrt{1-\eps} \le \eps$ yields
\[ |\mathcal{F}| \ge
\left( \frac{n}{8\log N}\right)^{\frac{n-1}{2}}.
\]
In fact, when $r$ approaches $1$, we still get a super-exponential lower bound. More precisely, let $P_n \subset \R^n$ be a sequence of polytopes and consider $\varepsilon(n)=\frac{\log N(P_{n},B_{2}^{n})}{n}$. By definition, the covering number is sub-exponential if and only if $\varepsilon(n)=o(1)$. Next consider the sequence of inradii $r(n)=1-\delta(n)$, and assume $\delta(n)=o(1)$, i.e. $r(n)$ tends to $1$. The lower bound of Theorem \ref{thm: main theorem} for the number of facets reads
\[
|\mathcal{F}| > \left( \frac{1}{2
\left(1-(1-\delta)\sqrt{1-4\varepsilon}\right)
} \right)^{\frac{n-1}{2}}\approx
\left( \frac{1}{2\delta + 4\varepsilon} \right)
^{\frac{n-1}{2}},
\]
which is of the order of $\max\{\delta,\varepsilon\}^{-\frac{n-1}{2}}$. In particular, the bound is a super-exponential if and only if $r(n)\to1$ and $\varepsilon(n)\to 0$.
\end{rem}

    
%

\bibliographystyle{amsplain}

\end{document}